\documentclass[a4paper]{article}
\usepackage[english]{babel}
\usepackage{amsmath,amssymb,amsthm,exscale,amsopn,mathtools,fullpage}
\usepackage{pstricks, pst-node, pst-text, pst-3d}
\usepackage{pstricks-add}
\usepackage{mathdots}

\newtheorem{lemma}{Lemma}[section]
\newtheorem{theorem}[lemma]{Theorem}
\newtheorem{obs}[lemma]{Observation}
\newtheorem{example}[lemma]{Example}

\DeclareMathOperator{\HK}{HK}
\DeclareMathOperator{\GK}{GKdim}
\DeclareMathOperator{\supp}{supp}
\DeclareMathOperator{\HH}{H}

\newcommand{\s}{\subseteq}
\newcommand{\ar}{\xrightarrow}
\numberwithin{equation}{section}

\newenvironment{proofy}{\noindent\textit{Proof of Lemma~\ref{rq}.}}{\hfill$\square$\medskip}
\newenvironment{proofx}{\noindent\textit{Proof of Theorem~\ref{main}.}}{\hfill$\square$\medskip}
\newenvironment{proofxx}{\noindent\textit{Proof of Lemma~\ref{sprim}.}}{\hfill$\square$\medskip}

\title{Gr\"obner basis and the automaton property of Hecke--Kiselman algebras}
\author{Arkadiusz M\c{e}cel and Jan Okni\'nski}
\date{}

\begin{document}
\maketitle

\begin{abstract}
It is shown that the Hecke-Kiselman algebra associated to a finite
directed graph is an automaton algebra in the  sense of
Ufnarovskii. Consequently, its Gelfand-Kirillov dimension is an
integer if it is finite. This answers a question stated in~\cite{mecel_okninski}. As a consequence, it is proved that the
Hecke-Kiselman algebra associated to an oriented cycle admits a
finite Gr\"obner basis.\\



\end{abstract}

\section{Introduction}

In the paper~\cite{maz} of Ganyushkin and Mazorchuk a finitely generated monoid $\HK_{\Theta}$
was defined for an arbitrary finite simple digraph $\Theta$ with $n$
vertices $\{1, \ldots, n\}$ by
specifying generators and relations.
\begin{itemize}
    \item[(i)] $\HK_{\Theta}$ is generated by idempotents $ x_i^2 = x_i$,
    where $1 \leq i \leq n$,
    \item[(ii)] if the vertices $i$, $j$ are not connected in  $\Theta$,
    then  $x_ix_j = x_jx_i$,
    \item[(iii)] if $i$, $j$ are connected by an arrow $i \to j$ in $\Theta$,
    then $x_ix_jx_i = x_jx_ix_j = x_ix_j$,
        \item[(iv)] if $i$, $j$ are connected by an (unoriented) edge in $\Theta$,
        then $x_ix_jx_i = x_jx_ix_j$.
\end{itemize}

\noindent If the graph $\Theta$ is unoriented (has no arrows), the
monoid  $\HK_{\Theta}$ is isomorphic to the so-called $0$-Hecke
monoid $\HH_0(W)$, where $W$ is the Coxeter group of the graph
$\Theta$, see~\cite{den}. The latter monoid plays an important
role in representation theory. In the case $\Theta$ is oriented (all
edges are arrows) and acyclic, the monoid $\HK_{\Theta}$ is finite
and it is a homomorphic image of the so-called Kiselman monoid
$K_n$, see~\cite{maz}, \cite{kun}. It is  worth mentioning that a characterization of general finite digraphs $\Theta$ such that the monoid $\HK_{\Theta}$ is finite remains an open problem, see \cite{aragona}.\\

\noindent The aim of this paper is to continue the study of the
semigroup algebra $A = k[\HK_{\Theta}]$ over a field $k$, in the
case when $\Theta$ is an oriented graph, that was started in~\cite{mecel_okninski}, where it was shown that the growth of $A$
is either polynomial or the monoid $\HK_{\Theta}$ contains a
noncommutative free submonoid. The main result of the present
paper states that the algebra $A$ is automaton in the sense of
Ufnarovskii~\cite{ufnar}, which means that the set of normal words
of $A$ forms a regular language. In other words, the set of normal
words of $A$ is determined by a finite automaton.

\begin{theorem}\label{main} Assume that $\Theta$ is a finite simple oriented graph. Then
$A = k[\HK_{\Theta}]$  is an automaton algebra, with respect
to any deg-lex order on the underlying free monoid of rank $n$.
Consequently, the Gelfand-Kirillov dimension $\GK(A)$ of
 $A$ is an integer if it is finite.
\end{theorem}

\noindent In the case when the digraph $\Theta$ is unoriented, the
corresponding monoid algebra is known to be automaton. Indeed, as
mentioned above: in this case $\HK_{\Theta} = \HH_0(W)$, where $W$
is the Coxeter group of the graph $\Theta$. In fact, one can prove
that the reduced words for $W$ and $\HH_0(W)$ are the same, and two
words represent the same element of the Coxeter group if and only
if they represent the same element of the Coxeter monoid, see~\cite{tsa}. However, the set of normal forms of elements of a
Coxeter group is known to be regular,
see~\cite{bri}.\\

\noindent We note that it was proved in~\cite{mecel_okninski}
that the following conditions are equivalent: 1) $k[\HK_{\Theta}]$
is a PI-algebra, 2) $\HK_{\Theta}$ does not contain a
noncommutative free submonoid, 3) $\GK( k[\HK_{\Theta}])$ is
finite, 4) $\Theta$ does not contain two different oriented cycles
connected by an oriented path. Theorem~\ref{main} answers a
question raised in~\cite{mecel_okninski}.\\

\noindent The key method used to obtain this result is the description of a Gr\"obner basis of Hecke-Kiselman algebras.
It is known that if the leading terms of the elements of this basis form a regular subset of the corresponding free monoid, then the algebra is automaton, see~\cite{ufnar}, Theorem~2 on page 97.  Consequently, our methods involve the monoid $\HK_{\Theta}$ only, rather than certain ring theoretical aspects of the algebra $k[\HK_{\Theta}]$. The obtained Gr\"obner basis is crucial for the approach to the structure of such algebras, which will be pursued in a forthcoming paper. \\

\noindent The class of automaton algebras was introduced by
Ufnarovskii in~\cite{ufn}. The main motivation was to study a
class of finitely generated algebras that generalizes the class of
algebras that admit a finite Gr\"obner basis with respect to some
choice of generators and an ordering on monomials. The difficulty
here lies in the fact that there are infinitely many generating
sets as well as infinitely many admissible orderings on monomials
to deal with. There are examples of algebras with finite Gr\"obner
bases with respect to one ordering, and infinite bases with
respect to the other. Up until recently it was not known whether
for any of known examples of automaton algebras with infinite
Gr\"obner bases with respect to certain orderings one could find a
better ordering that would yield a finite Gr\"obner basis.
First counterexamples were found by Iyudu and Shkarin in~\cite{iyu}.\\

\noindent There are many results indicating that the class of
automaton algebras not only has better computational properties
but also several structural properties that are better than in
the class of arbitrary finitely generated algebras. For example,
in this context one can refer to results on the Gelfand-Kirillov
dimension, results on the radical in the case of monomial
automaton algebras~\cite{ufn}, results on prime algebras of this
type~\cite{bell}, and also structural results concerned with the
special case of finitely presented monomial algebras~\cite{okn}.
In particular, finitely generated algebras of the following types
are automaton: commutative algebras, algebras defined by not more
than two quadratic relations, algebras for which all the defining
relations have the form $[x_ix_j] = 0$, for some pairs of
generators, see~\cite{ufnar}. Moreover, algebras that are finite
modules over commutative finitely generated subalgebras are also of this type~\cite{cedo_okninski}. Several aspects of automaton algebras have
been recently studied also in~\cite{iyu},  \cite{piontkovski1}, \cite{piontkovski2}. \\

\noindent In Section 2 we introduce the necessary definitions and
auxiliary results. Next, in Section 3, we determine a Gr\"obner
basis of $k[\HK_{\Theta}]$, from which the main result follows.
Finally, in Section 4, we prove that in the case when the graph
$\Theta$ is a cycle, $k[\HK_{\Theta}]$ has a finite  Gr\"obner
basis. An example is given to show that this is not true for
arbitrary Hecke-Kiselman algebras of oriented graphs, even in the
case when  the algebra satisfies a polynomial identity.

\section{Definitions and the necessary background}\label{dwa}

\noindent Let $F$ denote the free monoid on the set $X$ of $n\ge
3$ free generators $x_1,\dotsc,x_n$. Let $k$ be a field and let
$k[F] =k\langle x_1,\dotsc,x_n\rangle$ denote the corresponding
free algebra over $k$. Assume that a well order $<$ is fixed on
$X$ and consider the induced degree-lexicographical order on $F$
(also denoted by $<$). Let $A$ be a finitely
generated algebra over $k$ with a set of generators $r_1, \ldots,
r_n$ and let $\pi: k[F] \to A$ be the natural homomorphism of
$k$-algebras with $\pi(x_i) = r_i$. We will assume that $\ker
(\pi)$ is spanned by elements of the form $w-v$, where $w,v\in F$
 (in other words, $A$ is a semigroup algebra). Let $I$ be the
ideal of $F$ consisting of all leading monomials of $\ker(\pi)$.
The set of normal words corresponding to the chosen
presentation for $A$ and to the chosen order on $F$ is defined by
$N(A) = F \setminus I$.
One says that $A$ is an
automaton algebra if $N(A)$ is a regular language. That means that
this set is obtained from a finite subset of $F$ by applying a
finite sequence of operations of union, multiplication and
operation $*$ defined by $T^* = \bigcup_{i \geq 1}T^i$,
for $T \subseteq F$. If $T = \{w\}$ for some $w \in F$, then we write $T^* = w^*$.\\

\noindent For every $x \in X$ and $w \in F$ by $|w|_x$ we mean the
number of occurrences of $x$ in $w$. By $|w|$ we denote the length
of the word  $w$. The support of the word $w$, denoted by
$\supp(w)$, stands for the set of all $x \in X$ such that $|w|_x >
0$. We say that the word $w = x_1\cdots x_r \in F$ is a subword of
the word $v \in F$, where $x_i \in X$, if $ v = v_1x_1\cdots v_r
x_r v_{r+1}$, for some $v_1, \ldots, v_{r+1} \in F$. If $v_2,
\ldots, v_r$ are trivial words, then we say that $w$
is a factor of $v$. \\

\noindent Describing the normal words of a finitely generated
algebra $A$ is related to finding a Gr\"obner basis of the ideal $J=\ker(\pi)$.
 Recall that a subset $G$ of $J$ is called a Gr\"obner basis of
$J$ (or of $A$) if $0\notin G$, $J$ is generated by $G$ as an ideal and
for every nonzero $f\in J$ there exists $g\in G$ such that the
leading monomial $\overline{g}\in F$ of $g$ is a factor of the
leading monomial $\overline{f}$ of $f$.
If $G$ is a Gr\"obner basis of $A$, then a word $w\in F$ is normal
if and only if $w$ has no factors that are leading
monomials in $g\in G$. \\

\noindent The so-called diamond lemma is often used in this
context. We will follow the approach and terminology of~\cite{ber}. By a reduction in $k[F]$ determined by a pair
$(w,w')\in F^2$, where $w' < w$ (the deg-lex order of $F$), we
mean any operation of replacing a factor $w$ in a word $f \in F$
by the factor $w'$. For a set $T \subseteq F^2$ of such pairs
(these pairs will be called reductions as well) we say that the
word $f \in F$ is $T$-reduced if no factor of $f$ is the leading
term $w$ of a reduction $(w, w')$ from the set $T$. The deg-lex
order on $F$ satisfies the descending chain condition, which means
there is no infinite decreasing chain of elements in $F$. This
means that a $T$-reduced form of a word $w \in F$ can always
be obtained in a finite series of steps. The linear space
 spanned by $T$-reduced monomials in $k[F]$ is denoted by $R(T)$.\\

\noindent The diamond lemma gives necessary and sufficient conditions for the set $N(A)$ of normal words to coincide with the set of $T$-reduced words in $F$. The key tool is the
notion of ambiguity. Let $\sigma = (w_{\sigma}, v_{\sigma})$,
$\tau = (w_{\tau}, v_{\tau})$ be reductions in $T$. By an overlap
ambiguity we mean a quintuple $(\sigma, \tau, l, w, r)$, where $1
\neq l, w, r \in F$ are such that $w_{\sigma} = wr$ and $w_{\tau}
= lw$. A quintuple $(\sigma, \tau, l, w, r)$ is
called an inclusive ambiguity if $w_{\sigma} = w$ and $w_{\tau} =
lwr$. For brevity we will denote these ambiguities as $l(wr)
= (lw)r$ and $l(w)r = (lwr)$, respectively. We will also say that
they are of type $\sigma$-$\tau$. We say that the overlap
(inclusive, respectively) ambiguity is resolvable if
$v_{\tau}r$ and $lv_{\sigma}$ ($v_{\tau}$ and
$lv_{\sigma}r$, respectively) have equal $T$-reduced forms. We use the
following simplified version of Bergman's diamond lemma.

\begin{lemma}\label{diamond} Let $T$ be a reduction set in the
free algebra $k[F]$ over a field
$k$, with a fixed deg-lex order in the free monoid $F$ over
$X$. Then the following conditions are equivalent:
\begin{itemize}
\item all ambiguities on $T$ are resolvable,
\item each monomial $f \in F$ can be uniquely $T$-reduced,
\item if $I(T)$ denotes the ideal of $k[F]$ generated by $\{ w - v: (w, v) \in T \}$
then $k[F]=I(T) \oplus R(T)$ as vector spaces.
\end{itemize}
Moreover if the conditions above are satisfied then the
$k$-algebra $A = k[F]/I(T)$ can be identified with $R(T)$ equipped
with a $k$-algebra structure with $f \cdot g$ defined as the
$T$-reduced form of $fg$, for $f,g\in R(T)$. In this case, $\{ w - v: (w, v) \in
T \}$ is a Gr\"obner basis of $A$.
\end{lemma}

\section{Gr\"obner basis in the oriented graphs case}

\noindent In this section we will prove that for any oriented graph $\Theta = (V(\Theta), E(\Theta))$,
the language of normal words of the Hecke-Kiselman algebra $k[\HK_{\Theta}]$
is regular, and thus that the algebra is always automaton.\\

\noindent For $t \in V(\Theta)$ and $w \in F = \langle V(\Theta)\rangle$ we write $w
\nrightarrow t$ if $|w|_t = 0$ and there are no $x \in \supp(w)$
such that $x \rightarrow t$ in $\Theta$. Similarly, we define $t
\nrightarrow w$: again we assume that $|w|_t = 0$ and there is no
arrow $t \to y$, where $y \in \supp(w)$. In the case when $t \nrightarrow w$ and
$w \nrightarrow t$, we write $t \nleftrightarrow w$.

\begin{theorem}\label{basis} Let $\Theta$ be a finite simple oriented graph with
vertices $V(\Theta) = \{x_1, x_2, \ldots, x_n\}$. Extend the natural ordering $x_1 < x_2 < \cdots < x_n$
on the set $V(\Theta) $ to the deg-lex order on the free monoid $F = \langle V(\Theta)  \rangle$.
Consider the following set $T$ of reductions on the algebra $k[F]$:
\begin{itemize}
\item[(i)] $(twt, tw)$, for any $t \in V(\Theta)$ and $w \in F$ such that  $w \nrightarrow t$,
\item[(ii)] $(twt, wt)$, for any $t \in V(\Theta)$ and $w \in F$ such that  $t \nrightarrow w$,
\item[(iii)] $(t_1wt_2, t_2t_1w)$, for any $t_1, t_2 \in V(\Theta)$ and $w \in F$
   such that $t_1 > t_2$ and $t_2 \nleftrightarrow t_1w$.
\end{itemize}
\noindent Then the set $\{w - v, \text{ where } (w, v) \in T\}$
forms a Gr\"obner basis of the algebra $k[\HK_{\Theta}]$.
\end{theorem}

\begin{proof} Clearly, $w>v$ for every pair $(w,v)\in T$. Moreover, it is easy to
see that $w$ and $v$ represent the same
element of $\HK_{\Theta}$. It remains to use the diamond lemma. We
will prove that all overlap and inclusive ambiguities of the
reduction system $T$ are resolvable. We begin with a simple
observation.

\begin{obs}\label{tosamo} Assume that $t \in V(\Theta)$ and $w \in F$ are
such that $t \nleftrightarrow w$. Then the words $tw$ and $wt$
have equal $T$-reduced forms.
\end{obs}

\begin{proof} We argue by induction on the length $|w|$ of $w$. If $w = 1$,
the assertion is clear. If $w \in V(\Theta)$, then we either have
$tw \ar{(iii)} wt$ or $wt \ar{(iii)} tw$, since $t \nleftrightarrow w$. We proceed with the induction step. Assume that $w =
y_1 \cdots y_k$, where $y_i \in V(\Theta)$, for $i = 1, \ldots,
k$. If $y_1 > t$, then we apply (iii) to $wt$ and we are done. If
there exists $i > 1$ such that $y_i > t$, then $y_1\cdots y_kt
\ar{(iii)} y_1\cdots y_{i-1}ty_{i}\cdots y_k$. Now we apply the
induction hypothesis to the words $ty_1\cdots y_{i-1}$ and
$y_1\cdots y_{i-1}t$ and $T$-reduce them to some $w' \in F$. Thus
we get that $tw$ and $wt$ can be both $T$-reduced to
$w'y_{i}\cdots y_k$. Finally, if $y_i < t$, for all $i$, then by
using reduction (iii) $k$ times we get: $$ty_1\cdots y_k
\ar{(iii)} y_1ty_2 \cdots y_k \ar{(iii)} y_1y_2ty_3\cdots y_k
\ar{(iii)} \cdots \ar{(iii)} y_1\cdots y_kt.$$
\end{proof}

\noindent We will now list overlap and inclusive ambiguities of all possible
types (x)-(y) of pairs of reductions in~$T$, where (x), (y) $\in \{$(i), (ii), (iii)$\}$.\\

\noindent There are two overlap and one inclusive ambiguity of type (i)-(i):
\begin{enumerate}
\item $tw_1(tw_2t) = (tw_1t)w_2t$, for $t \in V(\Theta)$ and $w_1, w_2 \in F$ such that $w_1w_2 \nrightarrow t$,
\item $ t_1w_1(t_2w_2t_1w_3t_2) = (t_1w_1t_2w_2t_1)w_3t_2$, for $t_1, t_2 \in V(\Theta)$, $w_1, w_2, w_3 \in F$ such that $w_1t_2w_2 \nrightarrow t_1$ and $w_2t_1w_3 \nrightarrow t_2$,
\item $(t_1w_1t_2w_2t_2w_3t_1) = t_1w_1(t_2w_2t_2)w_3t_1$, for $t_1, t_2 \in V(\Theta)$, $w_1, w_2, w_3 \in F$ such that  $w_1t_2w_2t_2w_3 \nrightarrow t_1$ and $w_2 \nrightarrow  t_2$.
\end{enumerate}

\noindent There are two overlap and one inclusive ambiguity of type (i)-(ii):
\begin{enumerate}
\setcounter{enumi}{3}
\item $ tw_1(tw_2t) = (tw_1t)w_2t$, for $t \in V(\Theta)$, $w_1, w_2 \in F$ such that $t \nrightarrow w_1$, $w_2 \nrightarrow t$,
\item $ t_1w_1(t_2w_2t_1w_3t_2) = (t_1w_1t_2w_2t_1)w_3t_2$, for $t_1, t_2 \in V(\Theta)$ and $w_1, w_2, w_3 \in F$ such that $w_2t_1w_3 \nrightarrow t_2$, $t_1 \nrightarrow w_1t_2w_2$,
\item $(t_1w_1t_2w_2t_2w_3t_1) = t_1w_1(t_2w_2t_2)w_3t_1$, for $t_1, t_2 \in V(\Theta)$ and $w_1, w_2, w_3 \in F$  such that $w_1t_2w_2t_2w_3 \nrightarrow t_1$ and $t_2 \nrightarrow w_2$.
\end{enumerate}

\noindent There are two overlap and three inclusive ambiguities of type (i)-(iii):
\begin{enumerate}
\setcounter{enumi}{6}
\item $ t_1w_1(t_2w_2t_2) = (t_1w_1t_2)w_2t_2$, for $t_1, t_2 \in V(\Theta)$, $w_1, w_2 \in F$ such that $t_1 > t_2$, $w_2 \nrightarrow t_2$ and $t_2 \nleftrightarrow t_1w_1$,

\item $t_1w_1(t_2w_2t_3w_3t_2) = (t_1w_1t_2w_2t_3)w_3t_2$, for $t_1, t_2, t_3 \in V(\Theta)$, $w_1, w_2, w_3 \in F$ such that $t_1 > t_3,$ $w_2t_3w_3 \nrightarrow t_2$ and $t_3 \nleftrightarrow t_1w_1t_2w_2$,
\item $(t_1w_1t_2w_2t_1) = (t_1w_1t_2)w_2t_1$, for $t_1, t_2 \in V(\Theta)$, $w_1, w_2 \in F$ such that $t_1 > t_2$, $w_1t_2w_2  \nrightarrow t_1$ and $t_2 \nleftrightarrow t_1w_1$,
\item $(t_1w_1t_2w_2t_3w_3t_1) = t_1w_1(t_2w_2t_3)w_3t_1$, for $t_1, t_2, t_3 \in V(\Theta)$, $w_1, w_2, w_3 \in F$ such that $t_2 > t_3,$ $w_1t_2w_2t_3w_3 \nrightarrow t_1$ and $t_3 \nleftrightarrow t_2w_2$,
\item $(t_1w_1t_2w_2t_1) = t_1w_1(t_2w_2t_1)$, for $t_1, t_2 \in V(\Theta)$, $w_1, w_2 \in F$ such that $t_2 > t_1$, $w_1t_2w_2 \nrightarrow t_1$ and $t_1 \nleftrightarrow t_2w_2$.

\end{enumerate}

\noindent There are two overlap and one inclusive ambiguity of type (ii)-(i):
\begin{enumerate}
\setcounter{enumi}{11}
\item $ tw_1(tw_2t) = (tw_1t)w_2t$, for $t \in V(\Theta)$, $w_1, w_2 \in F$ such that $w_1 \not\to t$, $t \nrightarrow w_2$,
\item $ t_1w_1(t_2w_2t_1w_3t_2) = (t_1w_1t_2w_2t_1)w_3t_2$, for $t_1, t_2 \in V(\Theta)$ and $w_1, w_2, w_3 \in F$ such that $t_2 \nrightarrow w_2t_1w_3$, $w_1t_2w_2 \nrightarrow t_1$,
\item $(t_1w_1t_2w_2t_2w_3t_1) = t_1w_1(t_2w_2t_2)w_3t_1$, for $t_1, t_2 \in V(\Theta)$ and $w_1, w_2, w_3 \in F$  such that $t_1 \nrightarrow w_1t_2w_2t_2w_3$ and $w_2 \nrightarrow t_2$.
\end{enumerate}

\noindent There are two overlap and one inclusive ambiguity of type (ii)-(ii):
\begin{enumerate}
\setcounter{enumi}{14}
\item $tw_1(tw_2t) = (tw_1t)w_2t$, for $t \in V(\Theta)$ and $w_1, w_2 \in F$ such that $t \nrightarrow w_1w_2$,
\item $ t_1w_1(t_2w_2t_1w_3t_2) = (t_1w_1t_2w_2t_1)w_3t_2$, for $t_1, t_2 \in V(\Theta)$, $w_1, w_2, w_3 \in F$ such that $t_1 \nrightarrow w_1t_2w_2$ and $t_2 \nrightarrow w_2t_1w_3 $,
\item $(t_1w_1t_2w_2t_2w_3t_1) = t_1w_1(t_2w_2t_2)w_3t_1$, for $t_1, t_2 \in V(\Theta)$, $w_1, w_2, w_3 \in F$ such that  $t_1 \nrightarrow w_1t_2w_2t_2w_3$ and $t_2 \nrightarrow  w_2$.
\end{enumerate}

\noindent There are two overlap and three inclusive ambiguities of type (ii)-(iii):
\begin{enumerate}
\setcounter{enumi}{17}
\item $ t_1w_1(t_2w_2t_2) = (t_1w_1t_2)w_2t_2$, for $t_1, t_2 \in V(\Theta)$, $w_1, w_2 \in F$ such that $t_1 > t_2,$ $t_2 \nrightarrow w_2$ and $t_2 \nleftrightarrow t_1w_1$,

\item $t_1w_1(t_2w_2t_3w_3t_2) = (t_1w_1t_2w_2t_3)w_3t_2$, for $t_1, t_2, t_3 \in V(\Theta)$, $w_1, w_2, w_3 \in F$ such that $t_1 > t_3$, $t_2 \nrightarrow w_2t_2w_3$ and $t_3 \nleftrightarrow t_1w_1t_2w_2$,
\item $(t_1w_1t_2w_2t_1) = (t_1w_1t_2)w_2t_1$, for $t_1, t_2 \in V(\Theta)$, $w_1, w_2 \in F$ such that $t_1 > t_2$, $t_1 \nrightarrow w_1t_2w_2$ and $t_2 \nleftrightarrow t_1w_1$,
\item $(t_1w_1t_2w_2t_3w_3t_1) = t_1w_1(t_2w_2t_3)w_3t_1$, for $t_1, t_2, t_3 \in V(\Theta)$, $w_1, w_2, w_3 \in F$ such that $t_2 > t_3,$ $t_1 \nrightarrow w_1t_2w_2t_3w_3$ and $t_3 \nleftrightarrow t_2w_2$,
\item $(t_1w_1t_2w_2t_1) = t_1w_1(t_2w_2t_1)$, for $t_1, t_2 \in V(\Theta)$, $w_1, w_2 \in F$ such that $t_2 > t_1$, $t_1 \nrightarrow w_1t_2w_2$ and $t_1 \nleftrightarrow t_2w_2$.

\end{enumerate}

\noindent There are two overlap and two inclusive ambiguities of type (iii)-(i):
\begin{enumerate}
\setcounter{enumi}{22}

\item $t_1w_1(t_1w_2t_2) = (t_1w_1t_1)w_2t_2$, for $t_1, t_2 \in V(\Theta)$, $w_1, w_2 \in F$ such that $t_1 > t_2,$ $w_1 \nrightarrow t_1$ and $t_2 \nleftrightarrow t_1w_2$,
\item $t_1w_1(t_2w_2t_1w_3t_3) = (t_1w_1t_2w_2t_1)w_3t_3$, for $t_1, t_2, t_3 \in V(\Theta)$, $w_1, w_2, w_3 \in F$ such that $t_2 > t_3,$ $w_1t_2w_2 \nrightarrow t_1$ and $t_3 \nleftrightarrow t_2w_2t_1w_3$,
\item $(t_1w_1t_1w_2t_2) = (t_1w_1t_1)w_2t_2,$ for $t_1, t_2 \in V(\Theta)$, $w_1, w_2 \in F$ such that $t_1 > t_2$, $w_1 \nrightarrow t_1$ and $t_2 \nleftrightarrow t_1w_1t_1w_2$,
\item $(t_1w_1t_2w_2t_2w_3t_3) = t_1w_1(t_2w_2t_2)w_3t_3$, for $t_1, t_2, t_3 \in V(\Theta)$, $w_1, w_2, w_3 \in F$ such that $t_1 > t_3$, $w_2 \nrightarrow t_2$ and $t_3 \nleftrightarrow t_1w_1t_2w_2t_2w_3$.
\end{enumerate}

\noindent There are two overlap and two inclusive ambiguities of type (iii)-(ii):
\begin{enumerate}
\setcounter{enumi}{26}

\item $t_1w_1(t_1w_2t_2) = (t_1w_1t_1)w_2t_2$, for $t_1, t_2 \in V(\Theta)$, $w_1, w_2 \in F$ such that $t_1 > t_2,$ $t_1 \nrightarrow w_1$, $t_2 \nleftrightarrow t_1w_2$,
\item $t_1w_1(t_2w_2t_1w_3t_3) = (t_1w_1t_2w_2t_1)w_3t_3$, for $t_1, t_2, t_3 \in V(\Theta)$, $w_1, w_2, w_3 \in F$ such that $t_2 > t_3$, $t_1 \nrightarrow w_1t_2w_2$ and $t_3 \nleftrightarrow t_2w_2t_1w_3$,
\item $(t_1w_1t_1w_2t_2) = (t_1w_1t_1)w_2t_2,$ for $t_1, t_2 \in V(\Theta)$, $w_1, w_2 \in F$ such that $t_1 > t_2$, $t_1 \nrightarrow w_1$ and $t_2 \nleftrightarrow t_1w_1t_1w_2$,
\item $(t_1w_1t_2w_2t_2w_3t_3) = t_1w_1(t_2w_2t_2)w_3t_3$, for $t_1, t_2, t_3 \in V(\Theta)$, $w_1, w_2, w_3 \in F$ such that $t_1 > t_3$, $t_2 \nrightarrow w_2$ and $t_3 \nleftrightarrow t_1w_1t_2w_2t_2w_3$.
\end{enumerate}

\noindent There are two overlap and three inclusive ambiguities of type (iii)-(iii):
\begin{enumerate}
\setcounter{enumi}{30}
\item $ t_1w_1(t_2w_2t_3) = (t_1w_1t_2)w_2t_3$, for $t_1, t_2, t_3 \in V(\Theta)$, $w_1, w_2 \in F$ such that $t_1 > t_2$, $t_2 > t_3$, $t_3 \nleftrightarrow t_2w_2$ and $t_2 \nleftrightarrow t_1w_1$,
\item $t_1w_1(t_2w_2t_3w_3t_4) = (t_1w_1t_2w_2t_3)w_3t_4$, for $t_1, t_2, t_3, t_4 \in V(\Theta)$, $w_1, w_2, w_3 \in F$ such that $t_1 > t_3,$ $t_2 > t_4$, $t_4 \nleftrightarrow t_2w_2t_3w_3$ and $t_3 \nleftrightarrow t_1w_1t_2w_2$,
\item $(t_1w_1t_2w_2t_3) = (t_1w_1t_2)w_2t_3$, for $t_1, t_2, t_3 \in V(\Theta)$, $w_1, w_2 \in F$ such that $t_1 > t_2$, $t_1 > t_3$ and $t_3 \nleftrightarrow t_1w_1t_2w_2$, $t_2 \nleftrightarrow t_1w_1$,
\item $(t_1w_1t_2w_2t_3w_3t_4) = t_1w_1(t_2w_2t_3)w_3t_4$, for $t_1, t_2, t_3, t_4 \in V(\Theta)$, $w_1, w_2, w_3 \in F$ such that $t_1 > t_4$, $t_2 > t_3$, $t_4\nleftrightarrow t_1w_1t_2w_2t_3w_3$ and $t_3 \nleftrightarrow t_2w_2$,
\item $(t_1w_1t_2w_2t_3) = t_1w_1(t_2w_2t_3)$, for $t_1, t_2, t_3 \in V(\Theta)$. $w_1, w_2 \in F$ such that $t_1 > t_3, t_2 > t_3$, $t_3 \nleftrightarrow t_1w_1t_2w_2$.
\end{enumerate}

\noindent We will now solve these ambiguities.

\begin{enumerate}

\item $tw_1(tw_2t) \ar{(i)} tw_1tw_2 \ar{(i)} tw_1w_2$,\\
      $(tw_1t)w_2t \ar{(i)} tw_1w_2t \ar{(i)} tw_1w_2$.

\item $t_1w_1(t_2w_2t_1w_3t_2) \ar{(i)} t_1w_1t_2w_2t_1w_3 \ar{(i)} t_1w_1t_2w_2w_3,$\\
      $(t_1w_1t_2w_2t_1)w_3t_2 \ar{(i)} t_1w_1t_2w_2w_3t_2 \ar{(i)} t_1w_1t_2w_2w_3,$ since $w_2w_3 \nrightarrow t_2$.

\item $(t_1w_1t_2w_2t_2w_3t_1) \ar{(i)} t_1w_1t_2w_2t_2w_3 \ar{(i)} t_1w_1t_2w_2w_3$,\\
      $t_1w_1(t_2w_2t_2)w_3t_1 \ar{(i)} t_1w_1t_2w_2w_3t_1 \ar{(i)} t_1w_1t_2w_2w_3,$ since $w_1t_2w_2w_3 \nrightarrow  t_1$.

\item $tw_1(tw_2t) \ar{(i)} tw_1tw_2 \ar{(ii)} w_1tw_2,$\\
      $(tw_1t)w_2t \ar{(ii)} w_1tw_2t \ar{(i)} w_1tw_2$.

\item $t_1w_1(t_2w_2t_1w_3t_2) \ar{(i)} t_1w_1t_2w_2t_1w_3 \ar{(ii)} w_1t_2w_2t_1w_3,$\\
      $(t_1w_1t_2w_2t_1)w_3t_2 \ar{(ii)}w_1t_2w_2t_1w_3t_2 \ar{(i)}  w_1t_2w_2t_1w_3.$

\item $(t_1w_1t_2w_2t_2w_3t_1) \ar{(i)} t_1w_1t_2w_2t_2w_3 \ar{(ii)} t_1w_1w_2t_2w_3,$\\
      $t_1w_1(t_2w_2t_2)w_3t_1 \ar{(ii)}t_1w_1w_2t_2w_3t_1 \ar{(i)}  t_1w_1w_2t_2w_3,$ since $w_1w_2t_2w_3 \nrightarrow  t_1$.

\item $t_1w_1(t_2w_2t_2) \ar{(i)} t_1w_1t_2w_2 \ar{(iii)} t_2t_1w_1w_2$,\\
      $(t_1w_1t_2)w_2t_2 \ar{(iii)} t_2t_1w_1w_2t_2$. Since $t_2 \nleftrightarrow t_1w_1$, then $t_1w_1w_2 \nrightarrow t_2$ and we have:\\
            $t_2t_1w_1w_2t_2 \ar{(i)} t_2t_1w_1w_2.$

\item $t_1w_1(t_2w_2t_3w_3t_2) \ar{(i)} t_1w_1t_2w_2t_3w_3 \ar{(iii)} t_3t_1w_1t_2w_2w_3,$\\
      $(t_1w_1t_2w_2t_3)w_3t_2 \ar{(iii)} t_3t_1w_1t_2w_2w_3t_2 \ar{(i)} t_3t_1w_1t_2w_2w_3$, since $w_2w_3 \nrightarrow t_2$.

\item $(t_1w_1t_2w_2t_1) \ar{(i)} t_1w_1t_2w_2 \ar{(iii)} t_2t_1w_1w_2$,\\
      $(t_1w_1t_2)w_2t_1 \ar{(iii)} t_2t_1w_1w_2t_1 \ar{(i)} t_2t_1w_1w_2$, since $w_1w_2 \nrightarrow t_1$.

\item $(t_1w_1t_2w_2t_3w_3t_1) \ar{(i)} t_1w_1t_2w_2t_3w_3 \ar{(iii)} t_1w_1t_3t_2w_2w_3$,\\
      $t_1w_1(t_2w_2t_3)w_3t_1 \ar{(iii)} t_1w_1t_3t_2w_2w_3t_1 \ar{(i)} t_1w_1t_3t_2w_2w_3$, since $w_1t_3t_2w_2w_3 \nrightarrow t_1$.

\item $(t_1w_1t_2w_2t_1) \ar{(i)} t_1w_1t_2w_2$,\\
      $t_1w_1(t_2w_2t_1) \ar{(iii)} t_1w_1t_1t_2w_2 \ar{(i)} t_1w_1t_2w_2$, since $w_1 \nrightarrow t_1$.

\item $tw_1(tw_2t) \ar{(ii)} tw_1w_2t$,\\
      $(tw_1t)w_2t \ar{(i)} tw_1w_2t$.

\item $t_1w_1(t_2w_2t_1w_3t_2) \ar{(ii)} t_1w_1w_2t_1w_3t_2 \ar{(i)} t_1w_1w_2w_3t_2$, since $w_1w_2 \nrightarrow t_1$,\\
          $(t_1w_1t_2w_2t_1)w_3t_2 \ar{(i)} t_1w_1t_2w_2w_3t_2 \ar{(ii)} t_1w_1w_2w_3t_2$, since $t_2 \nrightarrow w_2w_3.$

\item $(t_1w_1t_2w_2t_2w_3t_1) \ar{(ii)} w_1t_2w_2t_2w_3t_1 \ar{(i)} w_1t_2w_2w_3t_1$,\\
      $t_1w_1(t_2w_2t_2)w_3t_1 \ar{(i)} t_1w_1t_2w_2w_3t_1 \ar{(ii)} w_1t_2w_2w_3t_1$, since $t_1 \nrightarrow w_1t_2w_2w_3$.

\item$tw_1(tw_2t) \ar{(ii)} tw_1w_2t \ar{(ii)} w_1w_2t$,\\
         $(tw_1t)w_2t \ar{(ii)} w_1tw_2t \ar{(ii)} w_1w_2t$.

\item $t_1w_1(t_2w_2t_1w_3t_2) \ar{(ii)} t_1w_1w_2t_1w_3t_2 \ar{(ii)} w_1w_2t_1w_3t_2$, since $t_1 \nrightarrow w_1w_2$,\\
      $(t_1w_1t_2w_2t_1)w_3t_2 \ar{(ii)} w_1t_2w_2t_1w_3t_2 \ar{(ii)} w_1w_2t_1w_3t_2$.

\item $(t_1w_1t_2w_2t_2w_3t_1) \ar{(ii)} w_1t_2w_2t_2w_3t_1 \ar{(ii)} w_1w_2t_2w_3t_1$,\\
      $t_1w_1(t_2w_2t_2)w_3t_1 \ar{(ii)} t_1w_1w_2t_2w_3t_1 \ar{(ii)} w_1w_2t_2w_3t_1$, since $t_1 \nrightarrow w_1w_2t_2w_3$.

\item $t_1w_1(t_2w_2t_2) \ar{(ii)} t_1w_1w_2t_2$,\\
      $(t_1w_1t_2)w_2t_2 \ar{(iii)} t_2t_1w_1w_2t_2$. Since $t_2 \nleftrightarrow t_1w_1$, we have $t_2 \nrightarrow t_1w_1w_2$ and thus:\\
            $t_2t_1w_1w_2t_2 \ar{(ii)} t_1w_1w_2t_2.$

\item $t_1w_1(t_2w_2t_3w_3t_2) \ar{(ii)} t_1w_1w_2t_3w_3t_2 \ar{(iii)} t_3t_1w_1w_2w_3t_2$, since $t_1 > t_3$ and $t_3 \nleftrightarrow  t_1w_1w_2$,\\
      $(t_1w_1t_2w_2t_3)w_3t_2 \ar{(iii)} t_3t_1w_1t_2w_2w_3t_2 \ar{(ii)} t_3t_1w_1w_2w_3t_2$, since $t_2 \nrightarrow w_2w_3$.

\item $(t_1w_1t_2w_2t_1) \ar{(ii)} w_1t_2w_2t_1$,\\
      $(t_1w_1t_2)w_2t_1 \ar{(iii)} t_2t_1w_1w_2t_1 \ar{(ii)} t_2w_1w_2t_1$. \\
            Since $t_2 \nleftrightarrow w_1$, we can use Observation~\ref{tosamo} to reduce $w_1t_2$ and $t_2w_1$ to the same form.

\item $(t_1w_1t_2w_2t_3w_3t_1) \ar{(ii)} w_1t_2w_2t_3w_3t_1 \ar{(iii)} w_1t_3t_2w_2w_3t_1$,\\
         $t_1w_1(t_2w_2t_3)w_3t_1 \ar{(iii)} t_1w_1t_3t_2w_2w_3t_1 \ar{(ii)} w_1t_3t_2w_2w_3t_1$.

\item $(t_1w_1t_2w_2t_1) \ar{(ii)} w_1t_2w_2t_1 \ar{(iii)} w_1t_1t_2w_2$,\\
      $t_1w_1(t_2w_2t_1) \ar{(iii)} t_1w_1t_1t_2w_2 \ar{(ii)} w_1t_1t_2w_2,$ since $t_1 \nrightarrow w_1$.

\item $t_1w_1(t_1w_2t_2) \ar{(iii)} t_1w_1t_2t_1w_2$. Since $t_2 \nleftrightarrow t_1w_2$, we have $w_1t_2 \nrightarrow t_1$ and:\\
      $t_1w_1t_2t_1w_2 \ar{(i)}   t_1w_1t_2w_2$. \\
      $(t_1w_1t_1)w_2t_2 \ar{(i)} t_1w_1w_2t_2$. Since $t_2 \nleftrightarrow w_2$, then by Observation~\ref{tosamo} we can reduce $t_2w_2$ and $w_2t_2$ to the same form.

\item $t_1w_1(t_2w_2t_1w_3t_3) \ar{(iii)} t_1w_1t_3t_2w_2t_1w_3$. Since $t_3 \nleftrightarrow t_2w_2t_1w_3$, then $w_1t_3t_2w_2 \nrightarrow t_1$ and:\\
      $t_1w_1t_3t_2w_2t_1w_3 \ar{(i)}   t_1w_1t_3t_2w_2w_3$.\\
            $(t_1w_1t_2w_2t_1)w_3t_3 \ar{(i)} t_1w_1t_2w_2w_3t_3 \ar{(iii)} t_1w_1t_3t_2w_2w_3$, since $t_3 \nleftrightarrow t_2w_2w_3$ and $t_2 > t_3$.

\item $(t_1w_1t_1w_2t_2) \ar{(iii)} t_2t_1w_1t_1w_2 \ar{(i)} t_2t_1w_1w_2$,\\
      $(t_1w_1t_1)w_2t_2 \ar{(i)} t_1w_1w_2t_2 \ar{(iii)} t_2t_1w_1w_2$, since $t_1 > t_2$ and $t_2 \nleftrightarrow t_1w_1w_2$.

\item $(t_1w_1t_2w_2t_2w_3t_3) \ar{(iii)} t_3t_1w_1t_2w_2t_2w_3 \ar{(i)} t_3t_1w_1t_2w_2w_3$,\\
      $t_1w_1(t_2w_2t_2)w_3t_3 \ar{(i)} t_1w_1t_2w_2w_3t_3 \ar{(iii)}    t_3t_1w_1t_2w_2w_3$, since $t_3 \nleftrightarrow t_1w_1t_2w_2w_3$ and $t_1 > t_3$.

\item $t_1w_1(t_1w_2t_2) \ar{(iii)} t_1w_1t_2t_1w_2$. Since $t_2 \nleftrightarrow w_2$, we have $t_1 \nrightarrow w_1t_2$, and thus:\\
      $t_1w_1t_2t_1w_2 \ar{(ii)} w_1t_2t_1w_2$.\\
            $(t_1w_1t_1)w_2t_2 \ar{(ii)} w_1t_1w_2t_2 \ar{(iii)} w_1t_2t_1w_2$.

\item $t_1w_1(t_2w_2t_1w_3t_3) \ar{(iii)} t_1w_1t_3t_2w_2t_1w_3$. Since $t_3 \nleftrightarrow t_2w_2t_1w_3$, we have $t_1 \nrightarrow w_1t_3t_2w_2$ and:\\
      $t_1w_1t_3t_2w_2t_1w_3 \ar{(ii)} w_1t_3t_2w_2t_1w_3.$\\
            $(t_1w_1t_2w_2t_1)w_3t_3 \ar{(ii)} w_1t_2w_2t_1w_3t_3 \ar{(iii)} w_1t_3t_2w_2t_1w_3$.

\item $ (t_1w_1t_1w_2t_2) \ar{(iii)}    t_2t_1w_1t_1w_2 \ar{(ii)} t_2w_1t_1w_2$,\\
      $(t_1w_1t_1)w_2t_2 \ar{(ii)} w_1t_1w_2t_2 \ar{(iii)} w_1t_2t_1w_2$, since $t_1 > t_2$ and $t_2 \nleftrightarrow t_1w_2$.
      Here, again we can see that $w_1 \nleftrightarrow t_2$ and thus $t_2w_1$ and $w_1t_2$ can be reduced to the same word, by Observation~\ref{tosamo}.

\item $(t_1w_1t_2w_2t_2w_3t_3) \ar{(iii)} t_3t_1w_1t_2w_2t_2w_3 \ar{(ii)} t_3t_1w_1w_2t_2w_3$,\\
        $t_1w_1(t_2w_2t_2)w_3t_3 \ar{(ii)} t_1w_1w_2t_2w_3t_3 \ar{(iii)} t_3t_1w_1w_2t_2w_3$, since $t_3 \nleftrightarrow t_1w_1w_2t_2w_3$ and $t_1 > t_3$.

\item $t_1w_1(t_2w_2t_3) \ar{(iii)} t_1w_1t_3t_2w_2 \ar{(iii)} t_2t_1w_1t_3w_2$, since $t_1 > t_2$ and $t_2 \nleftrightarrow t_1w_1t_3$, \\
            $(t_1w_1t_2)w_2t_3, \ar{(iii)} t_2t_1w_1w_2t_3$. \\
            Since $t_3 \nleftrightarrow w_2$ then by Observation~\ref{tosamo} $t_3w_2$ and $w_2t_3$ can be reduced to the same form.

\item $t_1w_1(t_2w_2t_3w_3t_4) \ar{(iii)} t_1w_1t_4t_2w_2t_3w_3 \ar{(iii)} t_3t_1w_1t_4t_2w_2w_3$, since $t_3 \nleftrightarrow t_1w_1t_4t_2w_2$ and $t_1 > t_3$,\\
      $(t_1w_1t_2w_2t_3)w_3t_4 \ar{(iii)} t_3t_1w_1t_2w_2w_3t_4 \ar{(iii)} t_3t_1w_1t_4t_2w_2w_3$, since $t_4 \nleftrightarrow t_2w_2w_3$ and $t_2 > t_4$.

\item $ (t_1w_1t_2w_2t_3) \ar{(iii)} t_3t_1w_1t_2w_2 \ar{(iii)} t_3t_2t_1w_1w_2$, since $t_1 > t_2$ and $t_2 \nleftrightarrow t_1w_1$,\\
      $ (t_1w_1t_2)w_2t_3 \ar{(iii)} t_2t_1w_1w_2t_3 \ar{(iii)} t_2t_3t_1w_1w_2$, since $t_1 > t_3$ and $t_3 \nleftrightarrow t_1w_1w_2$. \\
            Since $t_2 \nleftrightarrow t_3$, we either have $t_2t_3 \ar{(iii)} t_3t_2$, or $t_3t_2 \ar{(iii)} t_2t_3$.

\item $(t_1w_1t_2w_2t_3w_3t_4) \ar{(iii)} t_4t_1w_1t_2w_2t_3w_3 \ar{(iii)} t_4t_1w_1t_3t_2w_2w_3$,\\
      $t_1w_1(t_2w_2t_3)w_3t_4 \ar{(iii)} t_1w_1t_3t_2w_2w_3t_4 \ar{(iii)} t_4t_1w_1t_3t_2w_2w_3$, since $t_1 > t_4$ and $t_4 \nleftrightarrow t_1w_1t_3t_2w_2w_3$.

\item $(t_1w_1t_2w_2t_3) \ar{(iii)} t_3t_1w_1t_2w_2,$\\
      $t_1w_1(t_2w_2t_3) \ar{(iii)} t_1w_1t_3t_2w_2 \ar{(iii)} t_3t_1w_1t_2w_2$, since $t_1 > t_3$ and $t_3 \nleftrightarrow t_1w_1$.

\end{enumerate}
We have checked that all ambiguities of the reduction system $T$ are resolvable. Thus the diamond lemma can be applied and the result follows.
\end{proof}

\noindent We are ready to prove our first main result.\\

\begin{proofx}
We have
$$N_{\Theta} = N_{(i)} \cup N_{(ii)} \cup N_{(iii)},$$
where $N_{\Theta}$ stands for the set of leading terms in pairs
from the set $T$ considered in Theorem~\ref{basis}, and $N_{(i)}, N_{(ii)}, N_{(iii)}$ are the sets of leading terms from
the three families (i), (ii), (iii) of reductions in $T$, respectively. We only need to show that the sets $N_{(i)}, N_{(ii)}, N_{(iii)}$ are regular.\\

\noindent Indeed, observe that  $N_{(i)} = \{tvt\, | \, t \in V(\Theta), v \in F, v \nrightarrow t\} =
\bigcup\limits_{t \in V(\Theta)}\{tvt \mid v \nrightarrow t\}$ which is a finite union of sets of the
form $t\langle Y_t \rangle t$, where $Y_t$
is the subset of $V(\Theta)$ consisting of all generators $z$ such that
$z \nrightarrow t$. All these summands are clearly regular. Thus $N_{(i)}$ is regular. A similar argument works for $N_{(ii)}$.  \\

\noindent   Finally,
$$N_{(iii)} = \bigcup\limits_{x,z \in V(\Theta), x<z, x\nleftrightarrow z} z\langle X_x\rangle x,$$
where $X_x \subseteq V(\Theta)$ is the subset consisting of all generators $y \in V(\Theta)$
such that  $y \nleftrightarrow x$. Again, these summands are clearly
regular. Therefore, the set $N_{(iii)}$  is regular as a union
of regular sets.\\

\noindent As a result, the entire set $N_{\Theta}$ is regular and
it is well known that this implies that the algebra
$k[\HK_{\Theta}]$  is automaton, see~\cite{ufnar}, p. 97. The fact
that $\GK(k[\HK_{\Theta}])$ is an integer, if it is finite,
follows, see~\cite{ufnar}, Theorem~3 on page 97 and Theorem~1 on
page 90.
\end{proofx}

\section{Gr\"obner basis of a cycle monoid}\label{cycle}

\noindent Let $C_n$ denote the Hecke--Kiselman monoid associated
to the cycle consisting of $n\ge 3$ vertices. The aim of this
section is to prove that in the case of $k[C_n]$ one can find a
finite subset of the Gr\"obner basis obtained in the previous
section such that it itself forms a Gr\"obner basis of $k[C_n]$.
Our interest in this special case comes from the fact that the
structure of the algebra
$k[C_{n}]$ is crucial for the study of an arbitrary algebra $k[\HK_{\Theta}]$.\\

\noindent Recall that the monoid $C_n$ is
defined by generators $x_1,\dotsc,x_n$ subject to the following
relations:
   \begin{gather*}
        x_i^2=x_i,\\
        x_ix_{i+1}x_i=x_{i+1}x_ix_{i+1}=x_ix_{i+1},
    \end{gather*}
    for all $i=1,\dotsc,n$ (with the convention that indices are taken modulo $n$) and
   \[x_ix_j=x_jx_i\]
    for all $i,j=1,\dotsc,n$ satisfying $1<i-j<n-1$ (note that for $n=3$ there are no relations of this type).  \\

\noindent The natural order $x_1 < x_2 < \cdots < x_n$ is
considered on the set of generators and the corresponding deg-lex
order on the free monoid $F$. We also adopt the following notation
in this section. When we write a word of the form: $x_i \cdots
x_j$, we mean that consecutive generators from $x_i$ up to $x_j$
if $ i<j$ (or down to $x_{j}$, if $i
> j$) appear in this word. For instance, $x_2 \cdots x_5$ denotes
$x_2x_3x_4x_5$ and $x_6\cdots x_3$ stands
for the word $x_6x_5x_4x_3$.\\

\noindent Consider two sets $S$ and $S'$ of reductions on $k[F]$.
The first one is a subset of the system $T$ considered in the previous section that consists of all pairs of the form:
\begin{enumerate}
    \item[(1)] $(x_ix_i,x_i)$ for all $i\in\{1,\dotsc,n\}$,
    \item[(2)] $(x_jx_i,x_ix_j)$ for all $i,j\in\{1,\dotsc,n\}$ such that $1<j-i<n-1$,
    \item[(3)] $(x_n(x_1\dotsm x_i)x_j,x_jx_n(x_1\dotsm x_i))$ for all $i,j\in\{1,\dotsc,n\}$ such that $i+1<j<n-1$,
    \item[(4)] $(x_iux_i,x_iu)$ for all $i\in\{1,\dotsc,n\}$ and $1\ne u\in F$ such that $u \nrightarrow x_i$. Here, $i-1 = n$, for $i = 1$
    (we say, for the sake of simplicity, that the word $x_iux_i$ is of type $(4x_i)$),
    \item[(5)] $(x_ivx_i,vx_i)$ for all $i\in\{1,\dotsc,n\}$ and $1\ne v\in F$ such that $x_i \nrightarrow u$. Here $i+1 = 1$, for $i =n$
    (similarly, we say that the word $x_ivx_i$ is of type $(5x_i)$).
\end{enumerate}

 \noindent The second set of reductions is a
subset $S'$ of $S$ consisting of:
\begin{itemize}
    \item[(i)] all pairs of type (1)-(3),
    \item[(ii)] all pairs
$(x_iux_i,x_iu)$ of type (4) such that $|u|_{x_{j}}\le 1$, for
$j\in\{1,\dotsc,n\}\setminus\{i,i-1\}$,
\item[(iii)] all pairs
$(x_ivx_i,vx_i)$ of type (5) such that $|v|_{x_{j}}\le 1$ for $j\in\{1,\dotsc,n\}\setminus\{i,i+1\}$,
\item[(iv)] all pairs $(x_izx_i,zx_i)$ of type (5) such that $i < n$ and:
$$x_izx_i = x_i(x_{i_1}\cdots x_{j_1})(x_{i_2}\cdots x_{j_2}) \cdots (x_{i_k} \cdots x_{j_k})x_n(x_1\cdots x_i),$$
where $i_1 < i_2 < \cdots < i_k < n$ and $j_1 < j_2 < \cdots < j_k < n$.
\end{itemize}
We will say that the word $x_iux_i$ that appears in (ii) is of
type $(4x_{i}')$, the word $x_ivx_i$ that appears in (iii) is of type
$(5x_{i}')$, and the word $x_izx_i$ that appears in (iv) is of type
$(5x_{i}'')$. We will also say that a word $x \in F$ is of type (1),
(2), or (3),
respectively, if $x$ is the leading term of one of the reductions of the corresponding type.\\

\noindent One can recognize reductions of type (1) and (4) as
subsets of the reduction set (i) from Theorem~\ref{main}.
Similarly, reductions of types (2), (3) are special cases of
reductions of type (iii) and reductions of type (5) correspond to
the subset (ii) of $T$. It is convenient to explicitly distinguish
five families of reductions of the system $S$, as they will be
repeatedly used in the process of reducing the size of the
Gr\"obner basis obtained in the
previous section.\\

\noindent We will prove two facts concerning the reduction sets $S$ and $S'$.

\begin{lemma}\label{rq} Let $T$ be a reduction set on $k[C_n]$ obtained in Section 2. If $w \in F$ is $T$-reduced, then it is also $S$-reduced.
\end{lemma}

\begin{lemma}\label{sprim} Every $S'$-reduced word in $F$ is
$S$-reduced.

\end{lemma}

\noindent  The first lemma is a simple observation that is an
intermediate step towards the main result of this section. \\

\begin{proofy} Assume, to the contrary, that some word $w \in F$ is $S$-reduced, but not
$T$-reduced. Clearly, it is enough to consider the case where $w$ is of the form
(iii) from the definition of $T$, namely $v = x_kwx_i$,
where $k > i$ and $x_i \nleftrightarrow x_kw$. We will use inductive argument to show that $v$ is not $S$-reduced, which leads to a contradiction.\\

\noindent Of course, if $|w| = 0$ then $x_kx_i \ar{(2)} x_ix_k$,
so $x_kx_i$ is $S$-reducible. We proceed with the inductive step.
Let $|w| > 0$ and let $w = x_{i_1}\cdots x_{i_r}$, for some
$x_{i_s} \in \{x_1,\ldots, x_n \} $ such that $x_{i_s} \nleftrightarrow x_i$,
for $1 \leq s \leq r$. If for any $s$ we have $i_s > i$
then the factor $x_{i_s}\cdots x_{i_r}x_i$ is of the form (iii)
and thus it is not $S$-reduced, by the induction hypothesis.
So we only need to consider the case where $i_s \leq i < k$, for all $s$.
In particular, we have $i_1 \leq i < k$. We consider two cases.\\

\noindent Case 1. $k = n$. Here we must have $i_1 = 1$. Otherwise, an
$S$-reducible factor $x_nx_{i_1}$ appears in $v$ and the induction
step follows. If an $S$-reducible factor of the form (3) appears in $v$,
then we are done, so we may only consider the case where
$i_2 = 2, i_3 = 3, \ldots, x_r = r$. However, it follows that $r < i$,
since $i_s < i$, for all $s$. Since $x_i \nleftrightarrow w = x_nx_1\cdots x_r$,
we must have $i > r+1$, which means that $v$ is of the form (3) and $w$ is
thus $S$-reducible. The induction step follows again.\\

\noindent Case 2. $k < n$. In this case we either have $i_1 < k-1$ and
an $S$-reducible factor $x_kx_{i_1}$ appears in $v$, which yields the
induction step, or $i_1 = k - 1$. In the latter case we have
$ v = x_kx_{k-1}x_{i_2} \cdots x_rx_i$. However now we can repeat
the argument for $i_1$ to obtain that the only relevant case is
$i_2 = k-2.$ Indeed, we have $i_2 \neq k-1$, $i_2 \neq k$ and
$i_2 \leq i < k$. If we were to assume that $i_2 < k-2$, then
the $S$-reducible factor $x_{k-1}x_{i_2}$ would appear in $v$,
which would immediately yield the inductive step. After repeating
this process we are left with the case when
$v x_i= x_kx_{k-1}x_{k-2} \cdots x_m \cdot x_i$. However,
since $k > i$ and $x_i \nleftrightarrow x_kw$, we have  $k > m-1$,
so we get and an $S$-reducible factor $x_mx_i$. Thus, the induction step follows again.\\

\noindent We have shown that the word $v$ of the form (iii) is $S$-reducible,
which yields a contradiction. The assertion follows.
\end{proofy}

\noindent Before proving Lemma~\ref{sprim}, we will prove the following fact
 concerning certain special family of words.

    \begin{obs}\label{pe} Assume that $1 \neq p \in F$
     is such that $|p|_{x_{n}} = 0$ and $p$
    does not contain factors of the forms $(1)$-$(3)$, $(4x_i')$, $(5{x_i}')$, where $1 \leq i \leq n$.
    Then there exists $k \in \mathbb{N}$ such that $p$ is of the form:
    \begin{equation} (x_{i_1}\cdots x_{j_1})(x_{i_2}\cdots x_{j_2}) \cdots (x_{i_k} \cdots x_{j_k}), \label{formp}\end{equation}
    where $i_1 < i_2 < \cdots < i_k$ and $j_1 < j_2 < \cdots < j_k$, if $k > 1$.
    \end{obs}

\begin{proof} We need some additional notation.  We will say that a factor $v$ of a word $w
\in F$ is a block if $v$ is of the form $x_{i}\cdots x_{j}$, for
some $1 \leq i,j < n$, but there is no factor $v'$ of $w$ such
that $v$ is a factor of $v'$, the latter is also of the form
$x_{i'}\cdots x_{j'}$, for some $1 \leq i',j' < n$, and $v \neq
v'$. The length of a block $v$ is defined as the number $|j-i+1|$.
The block is called increasing if $i \leq j$
and decreasing if $i \geq j$ (note that $|p|_{x_{n}} = 0$).\\

\noindent Take $p \neq 1$ such that $|p|_{x_{n}} = 0$. Since $p$
cannot have subwords of the form $x_j x_{j+1}x_{j}$ or
$x_{j}x_{j-1}x_{j}$ (conditions $(4x_i')$, $(5{x_i}')$,
respectively), it follows that $p$ is (in a unique way) a product
of blocks and, by definition, the product of two consecutive
blocks is not a block. If $p$ is a product of an exactly one block
then there is nothing to prove -- $p$ is of the form~\eqref{formp}. Assume that $p$ is a product of at least two blocks
and take two consecutive blocks of the form $(x_{i_s}\cdots
x_{j_s})(x_{i_{s+1}}\cdots x_{j_{s+1}})$. Observe first, that we
cannot have $i_{s+1} \leq j_s +1$. Indeed, if $i_{s+1} <j_s  - 1$,
then a factor of type (2) would appear in $p$, a contradiction. If
we had $i_{s+1} = j_s \pm 1$, then either the product of the two
blocks $(x_{i_s}\cdots x_{j_s})(x_{i_{s+1}}\cdots x_{j_{s+1}})$ is
a block itself, or a factor of one of the forms
$x_{j_s}x_{j_{s}-1}x_{j_s}$, $x_{j_s}x_{j_{s}+1}x_{j_s}$ appears
in $p$, again a contradiction. Of course, we cannot have $i_{s+1}
= j_s$, as this yields a factor of type (1) in $p$.

We will prove that
$i_s < i_{s+1}$. Note that we cannot have $i_s = i_{s+1}$
 since this immediately
gives a factor $x_{i_s}\cdots x_{j_s}x_{i_{s}}$ of type
$(4x_{i_s}')$ or $(5x_{i_s}')$ in $p$, a contradiction. Assume, to
the contrary, that $i_{s+1} < i_{s}$. We already know that must
have $i_{s+1} > j_s +1$, so $j_s +1 < i_{s+1} < i_{s}$ and thus
the first block is decreasing of length $> 1$ and the factor of
the form $x_{i_{s+1}} \cdots x_{j_s} x_{i_{s+1}}$ of type
$(5x_{i_{s+1}}')$ appears in $p$, a contradiction. So $i_s <
i_{s+1}$. The inequality $j_s < j_{s+1}$ is proved in a completely
analogous way.  \end{proof}

\begin{proofxx}
Assume, to the contrary, that some word $w \in F$ is $S'$-reduced, but not
$S$-reduced. We may choose $w$ to be minimal with respect to the
deg-lex order on $F$. It is clear that $w$ may only be of the form
$(4x_i)$ or $(5x_i)$.\\

\noindent We will first consider the case $(4x_i)$; in other words
$w = x_iux_i$, for some $u \neq 1$, $|u|_{x_{i-1}} = |u|_{x_{i}} = 0$
(if $i = 1$, then $i-1 = n$).\\

\noindent First, observe that $i \neq n$. Indeed, if $i = n$, then
as $w$ is S'-reduced and $|u|_{x_{n}} = 0$, $u$ is of the form~\eqref{formp} and
\begin{equation} w = x_n(x_{i_1}\cdots x_{j_1})(x_{i_2}\cdots x_{j_2}) \cdots (x_{i_k} \cdots x_{j_k})x_n, \label{forman}\end{equation}
for some $k$ and $i_1 < i_2 < \cdots < i_k$ and $j_1 < j_2 <
\cdots < j_k$, if $k > 1$. As $|w|_{x_{n-1}} = 0$ and $x_nx_{i_1}$
cannot be of the form (2), we have $i_1 = 1$ and the first block
of $u$ is increasing. If $k > 1$, however, then $i_2 > j_1 + 1$,
since otherwise a factor $x_{i_2}\cdots x_{j_1}x_{i_2}$ of the
form ($4x_{i_2}'$) appears in $w$, which is impossible. But if $i_2
\neq n-1$, then $x_nx_{i_1}\cdots x_{j_1}x_{i_2}$ is a factor of
type (3) in $w$, a contradiction. Thus $k = 1$. In this case,
however, $w = x_n(x_{1}\cdots x_{j_1})x_n$
is of the form $(4x_{n}')$, again a contradiction. Therefore $i \neq n$.\\

\noindent Let $t = \max\{l: |u|_{x_{l}} \neq 0\}$. Of course, $t > 1$ as
otherwise $w$ is $S'$-reducible. Moreover,  $t > i$ since
otherwise $w$ has a prefix $x_ix_m$ with $m<i-1$, which is a word
of the form (2), a contradiction. We consider two cases: $1 < t
<n$ and $t=n$.
\begin{itemize}
    \item Case 1. $1 < t < n$. Since $w$ is $S'$-reduced and $|w|_{x_{n}} = 0$,
    then by Observation~\ref{pe} $w$ is of
    the form~\eqref{formp}, for some $k$ and $i_1 < \cdots < i_k$, $j_1 < \cdots < j_k$,
    if $k > 1$.
     But since $w$ is of the form $(4x_i)$, the first block of $w$ must begin with $x_i$,
     and the last block must end with $x_i$. If the length of the first block
    $x_i \cdots x_{j_1}$ was greater than 1, then this block must have been increasing,
    since $|w|_{x_{i-1}} = 0$. However,  in this case $i = i_1 < j_1 \leq j_k = i$, which is
    impossible. Thus the first block of $w$ consists just of $x_i$. If $k = 1$,
    then $w = x_i$, a contradiction. If $k > 1$, then $j_k > j_1 $, which is impossible,
    as $i_1 = j_1 < j_k = i$. Again, a contradiction.

\item Case 2. $t = n$. Then $i\neq 1$ because we are in the case $(4x_i)$.
Consider the last appearance of $x_n$ in $w$,
namely let $w = x_ipx_nqx_i$, where $p, q \in X$ and $|q|_{x_{n}} =
0$. First, assume that $q = 1$. Then $i$ must be equal to $n-1$
 since otherwise we would have a factor of type (2) in $w$.
 Hence $w = x_{n-1}px_nx_{n-1}$. If $p = 1$, then $w$ is of
type $(4x_{n-1}')$, which is impossible as $w$ is $S'$-reduced. Thus
$p = x_np'$, since otherwise $w$ contains a factor $x_{n-1}x_s$ of
type (2). Thus $w$ has a proper factor $x_np''x_n$ of type
$(4x_n)$, contradicting the minimality of the word $w$.
Thus we may assume that $q \neq 1$.

Since $w$ is $S'$-reduced, also $qx_i$ is $S'$-reduced and since
$|qx_i|_{x_{n}} = 0$, as $i < n$, we can apply Observation~\ref{pe} and
assume that it is of the form~\eqref{formp}, for some $k$ and $i_1
< \cdots < i_k$ and $j_1 < \cdots < j_k$, if $k > 1$. However,
since $x_nx_{i_1}$ is a factor of $w$ we must have $i_1 = n-1$ or
$i_1 = 1$, as otherwise $w$ has a factor of type (2). We consider
these subcases now:
\begin{itemize}
\item[(a)] If $i_1 = n-1$, then there is only one block in the decomposition~\eqref{formp} of $qx_i$,
otherwise another block of $qx_i$ would have to begin with $x_{i_2}$,
where $i_2 > i_1$ and also $n >i_2$. This is impossible.
 Therefore $w = x_ipx_nx_{n-1} \cdots x_i$.  If
$p = 1$ then $w$ is the form $(4x_{i}')$, a contradiction. Assume
that $p \neq 1$. Then $x_ipx_n\cdots x_{i+1}$ cannot contain two
occurrences of $x_{i+1}$ as that would yield a factor of the form
$(4x_{i+1})$ in $w$, which contradicts its minimality. Thus
$|x_ipx_n\cdots x_{i+2}|_{x_{i+1}} = 0$ and we can see that
$x_ipx_n\cdots x_{i+2}$ cannot contain two occurrences of
$x_{i+2}$. Continuing this way, we can see that $|p|_{x_{l}} = 0$,
for $n \geq l > i-1$. Thus $p = x_mp'$, for some $p'$, for some $m
< i-1$ and thus we have a factor $x_ix_m$ of type (2) in $w$, a
contradiction.
\item[(b)] If $i_1 = 1$, then $qx_i$ is of the form $(x_{1}\cdots x_{j_1}) \cdots (x_{i_k}
\cdots x_i)$. We cannot have $k = 1$, since in that case, we would
have a factor of the form $x_1\cdots x_i$ in $w$. Its length would
be greater than 1, since $i \neq 1$. Therefore, $w$ would contain
$x_{i-1}$, a contradiction. If $k > 1$ then as in the case of
words of the form~\eqref{forman} we have $i_2 = n-1$. This easily
implies that $k = 2$ and $w = x_ipx_n(x_{1}\cdots x_{j_1})(x_{n-1}
\cdots x_i)$. Next, if $p = 1$ then, since  $1 \leq  j_1 < i - 1$
the word $w$ is of the form $(4x_{i}')$, whence $w$ is $S'$-reducible, a
contradiction. Let $p \neq 1$. As in the previous subcase, we can
easily see that $|p|_{x_{l}}=0$ for $l=i+1, \ldots , n-1$. Again, if
$|p|_{x_{n}} \neq 0$, the minimality of $w$ is violated, and thus $|p|_{x_{n}} = 0$. Therefore $p =
x_mp'$, for some $m < i-1$. As in the previous case, a factor $x_ix_m$ of type
(2) appears in $w$, a contradiction.
\end{itemize}
\end{itemize}

\noindent We have proved that if $w$ is $S'$-reduced of the form
$(4x_i)$, then it is also $S$-reduced. Assume now
that $w$ is a minimal $S'$-reduced word of the form $(5x_i)$
 with respect to the deg-lex order on $F$. Namely, $w = x_iux_i$, for
some $u \neq 1$, $|u|_{x_{i+1}} = |u|_{x_{i}} = 0$. Formally, we need to note that $i+1 = 1$, if $i = n$,
but we will begin with showing that in fact $i \neq n$. \\

\noindent Assume the contrary, that $w = x_nux_n$ is of the form
$(5x_n)$. Thus $|u|_{x_{n}} = 0$ and by Observation~\ref{pe} $u$ must be of
the form~\eqref{formp}, for some $k$ and $i_1, \ldots, i_k$, $j_1,
\ldots, j_k$. Since $x_nx_{i_1}$ cannot be a factor of the form
(2), and $i_1 \neq 1$, we must have $i_1 = n-1$. This implies,
that $u$ is a product of only one block $x_{n-1}\cdots x_{j_1}$.
Since $j_1 > 1$
we can see that $w$ is of the form $(5x_n')$, a contradiction. Thus $i < n$.\\

\noindent Our approach will be similar to that from the first part
of the proof. Again, consider $t = \max\{l: |u|_{x_{l}} \neq 0\}$.
Clearly, $t > 1$, as otherwise $w$ is $S'$-reducible. The proof
breaks into two cases:

\begin{itemize}
\item Case 1. $t < n$. Since $w$ is $S'$-reduced and $|w|_{x_{n}} = 0$, it satisfies
the conditions of Observation~\ref{pe}. Thus it must be of the
form~\eqref{formp}, namely $w = (x_{i_1}\cdots
x_{j_1})(x_{i_2}\cdots x_{j_2}) \cdots (x_{i_k} \cdots x_{j_k}), $
    where $i_1 < i_2 < \cdots < i_k$ and $j_1 < j_2 < \cdots < j_k$, if $k > 1$.
    Of course, $w$ cannot consist of only one block $x_{i_1}\cdots x_{j_1}$, since otherwise
    we have $i_1 = j_1 = i$ and thus $w=x_{i}$, a contradiction. Hence $k >
    1$. We claim that $j_l \geq i$, for all $l > 1$. Indeed, if we had $j_l < i$, for some
     $1 < l \leq k$,
    then the block $x_{i_l}\cdots x_{j_l}$ would have to be decreasing, as $i_l > i_1 = i$,
    and thus it would contain $x_{i+1}$, a contradiction with the fact that $w$ is of the form $(5x_i)$.
    Thus $j_l \geq i$, for all $l > 1$.
    Consider the second block $x_{i_2}\cdots x_{j_2}$ of $w$.
     Of course $i_2 > i_1 = i$. If we had $j_2 = i$, then the entire second block
     of $w$ would be decreasing and it would contain $x_{i+1}$, a contradiction. So
     $j_2 > i$. It follows that $i = j_k \geq j_2 > i$, and we arrive
     at a contradiction, again.

\item Case 2. $ t= n$. Notice that $i\neq n-1$ in this case.
We assume, again, that $w = x_ipx_nqx_i$, where $|q|_{x_{n}} = 0$.
To avoid the appearance of a factor of type (2) in $w$, we must
restrict ourselves to one of the following subcases: (a) $q = 1$,
(b) $q = x_1q'$, or (c) $q = x_{n-1}q'$, where $q' \in F$.
\begin{itemize}
    \item Subcase (a). If $q = 1$, then $w = x_ipx_nx_i$. Therefore, $i = n-1$ or
    $i=1$, otherwise we have a factor of the form (2) in $w$. The first case was excluded
    in the beginning of Case 2. So $w = x_1px_nx_1$. Thus $p \neq 1$,  as otherwise $w$ is of type $(5x_1')$.
    Also, observe that $|p|_{x_{n}} = 0$, since otherwise we would have a proper factor $x_np'x_n$
    of $w$ such that $|p'|_{x_{n}} = |p'|_{x_{l}} = 0$ and thus this factor would be of the form $(5x_n)$.
    This violates the minimality of $w$ as a minimal $S'$-reduced and $S$-reducible word
    with respect to the deg-lex order in $F$. This means that $x_1p$ satisfies the conditions
    of Observation~\ref{pe} and is of the form~\eqref{formp}, so that $w$ is of the form
    $(5x_1'')$.
     This  contradicts the fact that it is $S'$-reduced.
\item Subcase (b). $q = x_1q'$. Again, $qx_i$ is of the form~\eqref{formp} and as $i < n-1$ it follows,
using the same arguments as in the case of words of the form~\eqref{forman}, that $qx_i$ must be a single block and thus $w =
x_ipx_nx_1\cdots x_i$. Now, by an argument used in the subcase (b)
of Case 2 in the first part of the proof,  when we considered
words $w$ of type $(4x_i)$, we can assume that $|p|_{x_{j}} =0$
for $j=i-1,i-2, \ldots ,1, n$. So $ |p|_{x_{n}} = 0$ allows us to
apply Observation~\ref{pe} to prove that $p$ is of the form~\eqref{formp}. This yields a contradiction, as $w$ is again proved
to be of the form $(5x_i'')$.
\item Subcase (c). $q = x_{n-1}q'$. {Once again, $qx_i$ is of the form~\eqref{formp}.
As the first block of $qx_i$ begins with $x_{n-1}$ we can see, as
before, that this is in fact the only block of this word.
Otherwise another block of $qx_i$ would have to begin with
$x_{i_2}$, where $i_2 > i_1 = n-1$ and also $n >i_2$. This is
impossible. Thus $qx_i = x_{n-1}q'x_i$ is a a single decreasing
block of length greater than~$1$ which is impossible, as
$|u|_{x_{i+1}} = 0$.}
\end{itemize}
The subcases (a)-(c) have been proved to lead to a
contradiction. Therefore, also in the case when $t = n$ we can see
that no $w$ can be $S'$-reduced but $S$-reducible.
\end{itemize}

\noindent So, every $S'$-reduced word is $S$-reduced.  Thus, Lemma~\ref{sprim} is proved.
\end{proofxx}\\

\noindent  It now follows easily from Lemma~\ref{rq} that the
reduction system $S$ satisfies the diamond lemma, because the
reduction system $T$ satisfies this. And similarly,
Lemma~\ref{sprim} implies then that the reduction system $S'$
satisfies the diamond lemma. Consequently, we have proved the
following theorem. 

\begin{theorem}\label{fingr}
$G'=\{w-w':(w,w')\in S'\}\s k[F]$ forms a finite
Gr\"obner basis and  $G=\{w-w':(w,w')\in S\}\s k[F]$
forms a Gr\"obner basis of the
algebra $k[\HK_{C_{n}}]$. Consequently, all $S'$-reduced words
form a basis of $k[\HK_{C_{n}}]$.
\end{theorem}

\noindent As mentioned before, the fact that in the
particular case of a cycle graph, even a finite Gr\"obner basis
can be obtained, strengthens the assertion of Theorem~\ref{main}
in view of~\cite{iyu}.\\

\noindent We conclude with an example showing that the above
result cannot be extended to arbitrary Hecke-Kiselman algebras of
oriented graphs, even in the case of PI-algebras.

\begin{example}
Let $\Theta$ be the graph obtained by adjoining an outgoing arrow
to the cycle $C_3$:

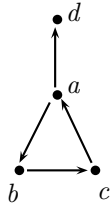
\begin{center}
\begin{pspicture}(6,3.5)(0,0.5)

\psdot(1,1) \rput(0.9,0.7){b}
\psdot(2,1)  \rput(2.1,0.7){c}
\psdot(1.5,2) \rput(1.7,2.1){a}
\psdot(1.5,3) \rput(1.7,3.1){d}
\psline{->}(1.1,1)(1.9,1)
\psline{<-}(1,1.1)(1.4,1.9)
\psline{->}(1.95,1.1)(1.55,1.95)
\psline{<-}(1.47,2.9)(1.47,2.1)
\end{pspicture}

\end{center}
then $V(\Theta )=\{a,b,c,d\}$ and we consider the
deg-lex order on the free monoid $F=\langle a,b,c,d\rangle$
defined by $a<b<c<d$. Then the algebra $k[\HK_{\Theta}]$ does not
have a finite Gr\"obner basis and it is a PI-algebra of
Gelfand-Kirillov dimension $2$.
\end{example}
\begin{proof}
It is easy to see that the set $N_{\Theta}$ used in the proof of
Theorem~\ref{main} is the union of the following subsets of $F$:
\begin{itemize}
   \item[(i)] $N_{(i)}= a\langle b,d\rangle a \cup b \langle c,d\rangle b
    \cup c\langle a,d\rangle c \cup d\langle b,c\rangle d $,
 \item[(ii)] $N_{(ii)}=a\langle
  c\rangle a \cup b\langle a,d\rangle b \cup c\langle b,d\rangle c
  \cup d\langle a,b,c\rangle d$,
 \item[(iii)] $N_{(iii)}= d\langle d \rangle b \cup  d\langle d \rangle c$.
\end{itemize}
 In particular, all words $d(abc)^kd, k\geq 1$, are in
$N_{(ii)}$, but they do not have factors that are another words of $N_{\Theta}$. 
It follows that the algebra $k[\HK_{\Theta}]$ does
not have a finite Gr\"obner basis (with respect to the indicated
presentation and deg-lex order). By~\cite{mecel_okninski}, this is
a PI-algebra. Moreover,  $k[C_{3}]$ is of linear growth with
reduced words being factors of two infinite words $(abc)^{\infty}$
and $(acb)^{\infty}$ and reduced words in $F$ are in the set
$F'\cup F'dF'$, where $F'=\langle a,b,c\rangle \subseteq F$. So,
$\GK(k[\HK_{\Theta}]) \leq 2$. On the other hand, words of the
form $(abc)^kd(abc)^m$, $k,m\geq 1$, do not have factors in
$N_{\Theta}$, so they are reduced. Therefore $\GK(k[\HK_{\Theta}]) = 2$.
\end{proof}

\noindent {\bf Acknowledgment.} The second author was partially supported by
the National Science Centre grant 2016/23/B/ST1/01045 (Poland).

\bigskip

\begin{tabular}{lll}
Arkadiusz M\c{e}cel & \quad \quad \quad \quad \quad & Jan Okni\'nski \\
\texttt{a.mecel@mimuw.edu.pl} & & \texttt{okninski@mimuw.edu.pl} \\
 & & \\
Institute of Mathematics & & \\
University of Warsaw & & \\
Banacha 2 & & \\
02-097 Warsaw, Poland & &
\end{tabular}

\begin{thebibliography}{99}
\itemsep=-2pt
\bibitem{aragona} Aragona R., Andrea A.D., Hecke-Kiselman monoids of small cardinality, Semigroup Forum 86 (2013), 32--40.
\bibitem{bell} Bell J., Colak P., Primitivity of finitely presented monomial algebras, J. Pure
Appl. Algebra 213 (2009), no. 7, 1299--1305.
\bibitem{ber} Bergman G.M., The diamond lemma
for ring theory, Advances in Mathematics  29 (1978), 178--218.
\bibitem{bri} Brink B., Howlett R.B., A finiteness property and an automatic structure for Coxeter groups, Math. Ann. 296 (1993), 179--190.
\bibitem{cedo_okninski} Ced\'o F., Okni\'nski J., On a class of
automaton algbras, Proc. Edinb. Math. Soc. 60 (2017), 31--38.
\bibitem{den} Denton T., Hivert F., Schilling A., Thiery N.M., On the representation theory of finite $J$-trivial monoids, Seminaire Lotharingien de Combinatoire 64 (2011), Art. B64d.
\bibitem{maz} Ganyushkin O., Mazorchuk V., On Kiselman quotients of $0$-Hecke monoids, Int. Electron. J. Algebra 10(2) (2011), 174--191.
\bibitem{iyu} Iyudu N., Shkarin S., Quadratic automaton algebras and intermediate growth, J. Comb. Algebra 2 (2018), 147--167.
\bibitem{kun} Kudryavtseva G., Mazorchuk V., On Kiselman's semigroup, Yokohama Math. J., 55(1) (2009), 21--46.
\bibitem{mecel_okninski} M\c{e}cel A., Okni\'nski J., Growth alternative for Hecke--Kiselman
monoids, Publicacions Matem\`atiques, to appear (2019).
\bibitem{okn} Okni\'nski J., Semigroup Algebras, Monographs and Textbooks in Pure and Applied Mathematics,
vol. 138, Marcel Dekker Inc., New York, 1991.
\bibitem{piontkovski1} Piontkovski D., Algebras of linear growth and the
dynamical Mordell--Lang conjecture, preprint, arXiv:1706.06470.
\bibitem{piontkovski2} Piontkovski D., Homogeneous finitely presented monoids of
linear growth, preprint, arXiv:1712.06022.
\bibitem{tsa} Tsaranov S.V., Representation and classification of Coxeter monoids, Eur. J. Comb 11
(1990), 189-204.
\bibitem{ufn} Ufnarovskii V.A., On the use of graphs for calculating the basis, growth and Hilbert series of
associative algebras, Math. Sb. 180 (1989), 1548--1560.
\bibitem{ufnar} Ufnarovskii V.A., Combinatorial and Asymptotic Methods
in Algebra, in: Encyclopedia of Mathematical Sciences vol. 57,
pp.1--196, Springer, 1995.
\end{thebibliography}
\end{document}